\documentclass[oneside,12pt]{amsart}
\usepackage{amssymb, amscd}
\usepackage[all]{xy}
\usepackage{enumerate}

\newtheorem{thm}{Theorem}[section]
\newtheorem{cor}[thm]{Corollary}
\newtheorem{lem}[thm]{Lemma}
\newtheorem{defn}[thm]{Definition}
\newtheorem{prop}[thm]{Proposition}

\newtheorem{conj}[thm]{Conjecture}

\newtheorem{remk}[thm]{Remark}

\newcommand{\del}[2]{{}}

\newcommand{\Z}{\mathbb Z}
\newcommand{\R}{\mathbb R}
\newcommand{\N}{\mathcal {N}}
\newcommand{\mP}{\mathcal {P}}
\newcommand{\mA}{\mathcal {A}}
\newcommand{\mX}{\mathcal {X}}

\title{New Action-Induced Nested Classes of Groups and Jump (Co)homology}

\author{Nansen Petrosyan}
\address{Department of Mathematics, Catholic University of Leuven, Kortrijk, Belgium}%
\email{Nansen.Petrosyan@kuleuven-kortrijk.be}%

\thanks{The author was supported by the Research Fund K.U.Leuven and FWO-Flanders Research Fellowship.}
\subjclass{}
\keywords{group actions, jump cohomology.}%

\date{\today}

\begin{document}
\begin{abstract}
Using fixed-point-free group actions, we set up a scheme to define nested classes of groups indexed over ordinals. Restricting to cellular
actions on CW-complexes, we find new classes as well as new characterizations for some well-known classes, such as virtually polycyclic
groups. We generalize properties of the virtual cohomological dimension of a group to groups with
jump (co)homology and prove that a core subclass of a new class of groups has jump (co)homology.
\end{abstract}

\maketitle
\tableofcontents

\maketitle

\section{Introduction}

When studying properties of a class of groups, constructing actions on topological spaces can give us important and often essential information.
Many types of groups are defined in this way.
One such example is Kropholler's
class of hierarchically defined  ${\sl H}\mathcal{F}$-groups. This class is defined inductively by cellular actions on finite dimensional
contractible CW-complexes (see Section 2 or \cite {krop}). It turns out to be a very large class,
containing all countable elementary amenable groups, all countable linear groups, and all groups with finite virtual cohomological dimension.

In section 2, we use similar but more general construction to define new, action-induced, nested classes of groups as follows.

\begin{defn}{\label{sc:01}}{\normalfont Let $\mX$ be a class of groups. Suppose $\mP$ is a condition on a space.
 Let $\mA$ be a restriction on the action of a group $G$
that acts on a space with property $\mP$ such that the induced action of each subgroup of $G$ on this space
also has the same restriction.

We define $\N(\mP, \mA, \mX)$ to be the smallest class of groups containing $\mX$
with the property that if a group $G$ acts by $\mA$ on a space with property $\mP$ such that all its isotropy groups are in $\N(\mP, \mA, \mX)$, then $G$ is also in $\N(\mP, \mA, \mX)$.}
\end{defn}

There is also hierarchical definition of $\N(\mP, \mA, \mX)$-groups, by using ordinal
numbers (see Section 2). We have that a group is in $\N(\mP, \mA, \mX)$ if and only if it belongs to  $\N_{\alpha}(\mP, \mA, \mX)$ for some
ordinal $\alpha$.

It is worth pointing out that although, we restrict ourselves to only cellular actions on finite dimensional CW-complexes, the definition allows one to consider many other types of actions, such as for example, algebraic actions on varieties, isometric actions on metric spaces, actions on manifolds and Lie groups, just to name a few.

To fix some notation, when $\mP \subset \{X| X \mbox{ is a finite dimensional CW-complex}\}$ and
$\mA$ defines the action to be cellular, we denote $\N(\mP, \mA, \mX)$ by $\N^{cell}(\mP, \mX)$ and when $\mX$ contains only the trivial group, by $\N^{cell}(\mP)$.

In section 3, we obtain the following descriptions for some well-known classes.

\begin{itemize}
\medskip
\item Let $\mP_1 = \{S^1\}$. Then, $\N^{cell}(\mP_1)$ is the class of finite solvable groups.
\medskip
\item Let $\mP_2 = \{\mathbb T^m| m\in \mathbb N\}$. Then, $\N^{cell}(\mP_2) = \N_1^{cell}(\mP_2)$ is the class of finite groups.
\medskip
\item Let $\mP_{3} = \{S^m| m\in \mathbb N\}$. Then, $\N^{cell}(\mP_{3})$ is the class of finite groups.
\medskip
\item Let $\mP_4 = \{S^1, \mathbb R\}$. Then, $\N^{cell}(\mP_4)$ is the class of polycyclic groups.
\medskip
\item Let $\mP_5 = \{S^m, \mathbb R | m\in \mathbb N\}$. Then, $\N^{cell}(\mP_5)$ is the class of virtually polycyclic groups.
\end{itemize}

%
%
%
%

In 1980, Olshanski disproved the Baer Conjecture by constructing a non-virtually polycyclic, namely simple and torsion-free, Noetherian group (see \cite{ol}). In the next theorem, we prove that a large class of groups closed under countable directed unions, HNN-extension and amalgamated products has a Tits alternative-like property and satisfies Baer criterion.

\begin{thm}{\label{sc:07}}{\normalfont  Let $\mP_6 = \{X| X=S^m, m\in \mathbb N, \mbox{ or } X  \mbox{ is a locally finite tree}\}$.
Then $\N^{cell}(\mP_6)$ contains all countable elementary amenable groups and all countable locally free groups. Every group in $\N^{cell}(\mP_6)$ either contains a free subgroup on two generators or it is countable elementary amenable. In particular, every Noetherian group in $\N^{cell}(\mP_6)$ is virtually polycyclic.}
\end{thm}

Next, we construct a class of groups, denoted by $\N^{cell}(\mP_R)$, which contains all ${\sl H}\mathcal{F}$-groups and it is the largest of all the classes we consider. It is defined by the following property.

\begin{defn}{\label{sc:08}}{\normalfont Suppose $R$ is an integral domain of characteristic zero. We say that a CW-complex $X$ belongs to $\mP_R$
whenever there exist $k\geq 0$ and  $m>0$ (both depending on $X$) such that
\begin{enumerate}[(a)]
\item $H_i(X)$ is $R$-torsion-free torsion group for each $i> k$,

\smallskip

\item $H_{k}(X)=\Z^m\oplus F$, where $F$ is an $R$-torsion-free finite group.
\end{enumerate}}
\end{defn}

These are a quite natural conditions on a CW-complex. For instance, when $R=\mathbb Q$, CW-complexes that have finitely generated homology groups, such as finitely dominated ones, satisfy both conditions.

In general, the subclass $\N^{cell}_1(\mP)$ consists of groups that can act freely and cellularly on a finite dimensional CW-complex with property $\mP$. So, it is quite easy to describe $\N^{cell}_1(\mP_i, \mX)$, $1\leq i\leq 6$.  By a theorem of Adem-Smith (see \ref{sc:3.14}),
every group with periodic cohomology acts freely and cellularly on a finite dimension CW-complex homotopy equivalent to a sphere. It follows then that every such group is in $\N^{cell}_1(\mP_R)$. In the last section, we show that all groups in $\N^{cell}_1(\mP_R)$ have jump cohomology over $R$.

In Section 4, we recall the definition of jump (co)homology and prove it has properties similar to  virtual (co)homological dimension of a group.

\begin{defn}{\label{sc:09}}{\normalfont Let $R$ be a commutative ring with a unit. A discrete group $G$ has
{\it jump (co)homology over $R$} if there exists an
integer $k\geq 0$, such that for each subgroup $H$ of $G$ we have
$hd_R(H)=\infty$ ($cd_R(H)=\infty$) or $hd_R(H)\leq k$
($cd_R(H)\leq k$).  The minimum of all such $k$ will
 be called {\it jump height} and denoted $hjh_R(G)$ ($cjh_R(G)$). }
\end{defn}

Since a group can have infinite torsion and still have jump (co)homology, many properties of virtual cohomological dimension
that hold for only virtually torsion-free groups extend naturally.
For instance, in Proposition \ref{sc:3.1}, we show that a group has jump (co)homology of height zero
if and only if it is all torsion. In Theorem \ref{sc:3.9}, we prove that a finitely generated solvable  group $G$ has finite
Hirsch length if and only if it has jump homology. In Theorem \ref{sc:3.12}, using a theorem of Alperin-Shalen, we prove that a linear group has jump homology if and only if there is an upper bound on the Hirsch lengths of its finitely generated unipotent subgroups.

\begin{thm}{\label{sc:015}}{\normalfont
Let $\mathcal{J_R}$ be the class of groups with jump cohomology over $R$
and let $\mathcal{VCD}$ denote the class of groups with finite virtual cohomological dimension.
Then, $$\mathcal{VCD}\subseteq \N^{cell}_1(\mP_R)\subseteq \mathcal{J_R}.$$}
\end{thm}

Lastly, as an application of the results in sections 4 and 5, we obtain the following

\begin{thm}{\label{sc:016}}{\normalfont Let $G$ be a solvable group and let $X$ be a finite dimensional
$G$-CW-complex. Suppose there exists an integer $t$ such that $\displaystyle{\bigoplus_{i\geq t} H_i(X)}$ is finitely generated and infinite.
Then $G$ has finite Hirsch length if and only if there is an upper bound on the Hirsch lengths of  all the stabilizer subgroups.}
\end{thm}

\section{The class of $\N(\mP, \mA, \mX)$-groups}

In \cite{krop}, Kropholler considered a fascinating class of hierarchically defined groups which he denoted by ${\sl H}\mathcal{X}$.
Let $\mathcal{X}$ be a class of groups. ${\sl
H}\mathcal{X}$ can be defined as the smallest class of groups containing
$\mathcal{X}$ with the property that if group $G$ acts cellularly on a
finite dimensional contractible CW-complex with all stabilizer
subgroups in ${\sl H}\mathcal{X}$, then $G$ is in ${\sl
H}\mathcal{X}$.

Let $\mathcal F$ denote the class of finite groups. In \cite{krop},  many properties  of ${\sl H}\mathcal{F}$-groups, such as
subgroup and extension closure, closure under countable direct
unions and free product, were established.
We will use a similar construction to define classes of groups inductively by their actions on topological spaces with a specified underlying
property.

\begin{defn}{\normalfont Let $\mX$ be a class of groups. Suppose $\mP$ is a condition on a space.
Let $\mA$ be a restriction on the action of a group $G$
that acts on a space with property $\mP$ such that the induced action of each subgroup of $G$ on this space
also has the same restriction.

We define $\N(\mP, \mA, \mX)$ to be the smallest class of groups containing $\mX$
with the property that if a group $G$ acts by $\mA$ on a space with property $\mP$ such that all its isotropy groups are in $\N(\mP, \mA, \mX)$, then $G$ is also in $\N(\mP, \mA, \mX)$.}
\end{defn}

\begin{remk}{\normalfont The condition that $X$ satisfies a given property $\mP$ is equivalent to requiring  that it belongs to a chosen set $\mP$ of topological spaces.}
\end{remk}

If $\mX$ is the class of finite groups, $\mP=\{X| X \mbox{ is a finite dimensional contractible CW-complex}\}$, and $\mA$ states that
the action is cellular, then the class $\N(\mP, \mA, \mX)$ is exactly the class of  ${\sl H}{\mathcal F}$-groups.

As with ${\sl H}{\mathcal F}$-groups, there exists an inductive definition of $\N(\mP, \mA, \mX)$-groups, using ordinal
numbers as follows:

\begin{enumerate}[(a)]
\item Let $\N_0(\mP, \mA, \mX)=\mX$.

\smallskip

\item For ordinal $\beta >0$, define $\N_{\beta}(\mP, \mA, \mX)$ to be the class of groups
that can act by $\mA$ on a space $X\in \mP$
such that each isotropy group is in $\N_{\alpha}(\mP, \mA, \mX)$ for some $\alpha < \beta$ ($\alpha$
can depend on the isotropy group).
\end{enumerate}

\smallskip

Clearly, a group is in $\N(\mP, \mA, \mX)$ if and only if it is in  $\N_{\alpha}(\mP, \mA, \mX)$ for some $\alpha$.


\begin{lem}{\label{sc:1}}{\normalfont Let $\mX$ be a subgroup closed class of groups. Then $\N(\mP, \mA, \mX)$ is subgroup closed.
  In addition, if $\N(\mP, \mA, \mX)\mX = \N(\mP, \mA, \mX)$, then $\N(\mP, \mA, \mX)$ is also extension closed.}
\end{lem}

\begin{proof} We will only prove extension closure. The proof for subgroup closure is similar and straightforward.
%

Suppose we have a short exact sequence of groups $$1\rightarrow K\longrightarrow G \longrightarrow Q\rightarrow 1,$$ such that $K\in \N_{\alpha}(\mP, \mA, \mX)$
and $Q\in \N_{\beta}(\mP, \mA, \mX)$. We will use transfinite induction on the ordinal $\beta$ to show that $G\in \N_{\alpha + \beta}(\mP, \mA, \mX)$.

We start with $\beta=0$ as the trivial case. Suppose $\beta>0$ and assume that for each $\gamma < \beta$, if
$Q\in \N_{\gamma}(\mP, \mA, \mX)$, then $G\in \N_{\alpha + \gamma}(\mP, \mA, \mX)$. By definition, $Q$ acts on a space $X\in \mP$
such that each isotropy subgroup is in $\N_{\theta}(\mP, \mA, \mX)$ for some $\theta < \beta$. Using the given epimorphism of $G$ onto $Q$,
we can construct an induced action of $G$ on $X$. The isotropy subgroups of this action are extensions of $K$ by the isotropy subgroups
of $Q$. By induction, each of these groups is in $\N_{\alpha + \theta}(\mP, \mA, \mX)$ for some $\theta < \beta$.
This shows that $G\in \N_{\alpha + \beta}(\mP, \mA, \mX)$, as claimed.

\end{proof}

The hypothesis of the lemma are clearly satisfied if we take $\mX$ as the class containing only the trivial group.
In this case, we denote  $\N(\mP, \mA, \mX)$  by $\N(\mP, \mA)$. Another class of groups
that satisfies these hypothesis is ${\sl H}\mathcal{F}$ (see 2.2 of \cite{krop}).


There is a natural partial ordering on the set of triples $(\mP, \mA, \mX)$ defining
$\N(\mP, \mA, \mX)$-classes.

\begin{defn}{\normalfont Let $(\mP, \mA, \mX)$ and $(\mP', \mA', \mX')$ be the triples defining
the groups $\N(\mP, \mA, \mX)$ and $\N(\mP', \mA', \mX')$, respectively. We say $(\mP, \mA, \mX)\leq (\mP', \mA', \mX')$,
 if $\mP\subseteq \mP'$, every action with restriction $\mA$ also has restriction $\mA'$,
and $\mX\subseteq \mX'$.}
\end{defn}

\noindent Clearly, when $(\mP, \mA, \mX)\leq (\mP', \mA', \mX')$, then $\N_{\alpha}(\mP, \mA, \mX)\subseteq \N_{\alpha}(\mP', \mA', \mX')$
for each ordinal $\alpha$ and $\N(\mP, \mA, \mX)\subseteq \N(\mP', \mA', \mX')$.

\medskip

\section{Cellular Actions}
We will assume throughout that whenever a group acts cellularly on a CW-complex,
then the action of the stabilizer group of any cell fixes that cell pointwise.

In this section, we consider classes defined using cellular actions on finite dimensional CW-complexes,
$\mP \subset \{X| X \mbox{ is a finite dimensional CW-complex}\}$ and
$\mA$ defines the action to be cellular. We denote this class by $\N^{cell}(\mP, \mX)$ and when $\mX$ is trivial, by $\N^{cell}(\mP)$.

\begin{prop}{\label{sc:2}}{\normalfont Let $\mP_1 = \{S^1\}$. Then, $\N^{cell}(\mP_1)$ is the class of finite solvable groups.}
\end{prop}

\begin{proof} First, we will show that the class of finite solvable groups is in $\N^{cell}(\mP_1)$.

Suppose $G$ is a finite solvable group. We proceed by induction on the length $n$ of the given decomposition
of $G$ into cyclic factors.

 Let $H$ be a normal subgroup of length $n-1$. Since $G/H$ acts freely and
 cellularly on $S^1$, from extension closure, we have
that $G\in \N^{cell}(\mP_1)$.

Conversely, suppose $G\in \N^{cell}(\mP_1)$. By transfinite induction, we can assume that $G$ acts cellularly on a
circle such that each stabilizer subgroup is finite solvable. The following
easy lemma finishes the proof.
\end{proof}

\begin{lem}{\label{sc:3}} {\normalfont A group $G$  acts effectively and cellularly on $S^1$ if and only if
 it is a subgroup of a dihedral group.}
\end{lem}

\begin{prop}{\label{sc:4}}{\normalfont Let $\mP_2 = \{\mathbb T^m| m\in \mathbb N\}$. Then, $\N^{cell}(\mP_2) = \N_1^{cell}(\mP_2)$ is the class of finite groups.}
\end{prop}

\begin{proof} We will show that the class of finite groups is inside $\N_1^{cell}(\mP_2)$, i.e. any finite group acts freely on a torus.

Suppose $G$ is a finite group. Let $g\in G$ and let $H=\langle g \rangle$. Construct a free and cellular
action of $H$ on $S^1$.
Let $X_g=\mbox{Map}_{H}(G, S^1)$ be the set of $H$-equivariant
set maps from $G$ to $S^1$ where $H$ acts on $G$ by left translation.  Let $g_1, g_2, \dots , g_n$ be coset representatives of $G/H$. There exists a bijection
$$\phi : X_g\rightarrow \prod_{i=1}^n S^1,$$
given by evaluating a map in $X_g$ at $g_i$ for each $1\leq i\leq n$. We now give $X_g$ the topology and CW-structure
coming from the product on the right hand side via the bijection $\phi$.
This structure is independent of the coset representatives. We define the action of $G$ on $X_g$
by $(x_0 f)(x)=f(xx_0)$ for $f\in X_g$, $x_0, x\in G$. Since $H$ acts freely on $X$,
the intersection of each stabilizer group of the action of $G$ on $X_g$ with $H$ is trivial.

Let $X= \prod_{g\in G} X_g$. Then $G$ act freely on $X$, because any stabilizer subgroup
is the intersection of the stabilizer subgroups of the actions of $G$ on $X_g$ for all $g\in G$.

It is left to show that any group in $\N^{cell}(\mP_2)$ is finite. But if  $G\in \N^{cell}(\mP_2)$,
again using transfinite induction, it follows that $G$ acts cellularly on a torus
with all stabilizer subgroups finite. Hence, it is finite.
\end{proof}


We note that the property $\mP_2$ in the theorem can be weakened.
We can instead consider finite CW-complexes homotopy equivalent to a tori. The same
conclusion will still hold.

Our next result shows that the class of finite groups can also be obtained by replacing tori with spheres.

\begin{prop}{\label{sc:6}}{\normalfont Let $\mP_{3} = \{S^m| m\in \mathbb N\}$. Then, $\N^{cell}(\mP_{3})$ is the class of finite groups.}
\end{prop}

\begin{proof} By induction on the ordinals, it follows that any group in  $\N^{cell}(\mP_{3})$ is finite.

In the following lemma, we will show that a group $G$ of order $n$ acts on $S^{n-2}$ without fixed points. Illman's
theorem would then shows that $S^{n-2}$ has a $G$-invariant triangulation, and therefore a $G$-CW-complex structure (see \cite{illman}). By induction
on the order of the group, we can immediately conclude that $G$ is in $\N^{cell}(\mP_{3})$.
\end{proof}

\begin{lem}{\label{sc:7}}{\normalfont Let $G$ be a finite group. Then $G$ acts isometrically on $S^{|G|-2}$ without fixed points.}
\end{lem}

\begin{proof}Let $n=|G|$ and consider the group ring $\mathbb RG$. Define $$V= \{\sum^n_{i=1} x_i g_i | \sum^n_{i=1} x_i=0\}.$$ Then
$V$ is a hyperplane of dimension $n-1$.  The space of vectors in $V$ of unit length is then a sphere $S^{n-2}$. Note that $G$ acts
isometrically on $V$ and hence on $S^{n-2}$. The resulting action is fixed-point-free.
%
\end{proof}

\begin{prop}{\label{sc:8}}{\normalfont Let $\mP_4 = \{S^1, \mathbb R\}$. Then, $\N^{cell}(\mP_4)$ is the class of polycyclic groups.}
\end{prop}

\begin{proof}
%
%
By extension closure, it follows that $\N^{cell}(\mP_4)$ contains all polycyclic group.

For the converse, suppose $G\in \N^{cell}(\mP_4)$. Then $G$ acts cellularly on $S^1$ or $\R$ such that each stabilizer subgroup is polycyclic.
If $G$ acts on $S^1$, then by Lemma \ref{sc:3}, it follows that $G$ is polycyclic. If $G$ acts on $\R$, the next lemma finishes the proof.
\end{proof}

\begin{lem}{\label{sc:9}}{\normalfont A group $G$  acts effectively and cellularly on $\R$ if and only if
$G$ is a subgroup of an infinite dihedral group.}
\end{lem}

\begin{thm}{\label{sc:10}}{\normalfont Let $\mP_5 = \{S^m, \mathbb R | m\in \mathbb N\}$. Then, $\N^{cell}(\mP_5)$ is the class of virtually polycyclic groups.}
\end{thm}

\begin{proof}  Since $\{\mP_{3}, \mP_4\}\leq \mP_5$,
we have $\{\N^{cell}(\mP_{3}), \N^{cell}(\mP_4)\}\subseteq \N^{cell}(\mP_5)$. Then, $\N^{cell}(\mP_5)$ contains all finite, all polycyclic, and
hence all virtually polycyclic groups.

Conversely, suppose $G\in \N_{\beta}^{cell}(\mP_5)$ for some ordinal $\alpha > 0$. Proceeding by transfinite induction, we can assume that
 $G$ acts cellularly on $S^n$ for some $n\geq 1$ or on $\R$ with all stabilizer subgroups virtually polycyclic. Without loss of generality,
we can assume that the action is effective. Then, Lemma \ref{sc:9} finishes the proof.
\end{proof}

In our next theorem, we prove that there  is a large class of groups which is closed under
countable direct unions, HNN-extensions, and amalgamated products such that virtually polycyclic groups are the only groups inside this class
that are Noetherian.

\begin{thm}{\label{sc:11}}{\normalfont  Let $\mP_6 = \{X| X=S^m, m\in \mathbb N, \mbox{ or } X  \mbox{ is a locally finite tree}\}$.
Then $\N^{cell}(\mP_6)$ contains all countable elementary amenable groups and all countable locally free groups. Every group in $\N^{cell}(\mP_6)$
either contains a free subgroup on two generators or it is countable elementary amenable. In particular, every Noetherian group in $\N^{cell}(\mP_6)$ is    virtually polycyclic.}
\end{thm}

\begin{proof}
Since $\N^{cell}(\mP_6)$ contains all finite and all finitely generated abelian groups, Lemma \ref{sc:13} shows that
it also contains all countable elementary amenable groups and all countable locally free groups.


Now, we suppose $G\in \N^{cell}_{\alpha}(\mP_6, \mX)$  for some ordinal $\alpha > 0$, does not contain a free subgroup on two generators.
We will show that $G$ is countable elementary amenable.

By transfinite induction, we can assume that $G$ acts cellularly on a CW-complex $X$, a sphere or a locally finite tree, where
 all stabilizer subgroups are countable elementary amenable.

Suppose $X$ is a sphere. Let $v$ be a $0$-cell of $X$. Clearly, $G/G_v$ is finite and hence, $G$ is countable elementary amenable.

Let $X$ be  a locally finite tree.  We can assume that the action of $G$ on $X$ is effective.
Since $G$ does not contain a free group on two generators, by a theorem of Pays and Valette (see \cite{PV}), $G$ either fixes a vertex, an edge, an end, or pair of ends. If $G$ fixes a vertex or an edge, then $G$ is countable elementary amenable by induction. If $G$ fixes a pair of ends then there exists a subgroup of index at most two fixing one of the two ends. Therefore, it is left to prove the case where $G$ fixes an end.

Suppose $\omega$ is an end fixed by $G$. Then for any $g\in G$ and  a ray $r\in \omega$,
$gr\in \omega$.

We recall a homomorphism $f:G \to \Z$ defined by $$f(g)= d(gx, x_0) - d (x,x_0),$$ where $x_0$ is a fixed vertex on a ray $r$ of $\omega$ and $x$ is a point on $r$ between $x_0$ and infinity such that $gx$ is also on $r$. One can easily check that the definition is independent of the choice of $x_0$ and $x$.

Let $K$ be the kernel of $f$. It is left to show that $K$ is countable elementary amenable. Let $r\in \omega$ and let $x_i$, $i\in \mathbb N$, be vertices of $r$ such that $x_{i+1}$ is between $x_i$ and infinity for each $i$. Let $$K_i=\{g\in K| g(x_j)=x_j, \mbox{ for each } j\geq i\}.$$ We have that for each $i$, $K_i$ is a group and $K_i\leq K_{i+1}$. It follows that $K= \cup_{i\in \mathbb N} K_i$. As $K_i$ is a subgroup of $G_{x_i}$ for each $i\in \mathbb N$, it is countable elementary amenable. Therefore, $K$ is countable elementary amenable.
%
\end{proof}


%
%

Next, we define a class with the most general property $\mP$ amongst all classes of groups
we have considered thus far.

\begin{defn}{\label{sc:12}}{\normalfont Suppose $R$ is an integral domain of characteristic zero. We say that a CW-complex $X$ belongs to $\mP_R$
whenever there exist $k\geq 0$ and  $m>0$ (both depending on $X$) such that
\begin{enumerate}[(a)]
\item $H_i(X)$ is $R$-torsion-free torsion group for each $i> k$,

\smallskip

\item $H_{k}(X)=\Z^m\oplus F$, where $F$ is an $R$-torsion-free finite group.
\end{enumerate}}
\end{defn}

\smallskip

We point out that when $R=\mathbb Q$, complexes that have finitely generated  homology groups, such as finitely dominated CW-complexes, satisfy these conditions. As we shall see in the last section, all groups in $\N^{cell}_1(\mP_R)$ will have jump cohomology over $R$.

Next, we make several observations.


\begin{remk}{\normalfont  Since $\mP_1\leq \{\mP_2, \mP_3, \mP_4\}$ and $\{\mP_3, \mP_4\}\leq \mP_5\leq \mP_6\leq \mP_R$, for
each ordinal $\alpha$
 we have $$\N^{cell}_{\alpha}(\mP_1)\leq \{\N^{cell}_{\alpha}(\mP_2), \N^{cell}_{\alpha}(\mP_3), \N^{cell}_{\alpha}(\mP_4)\}$$ and $$\{\N^{cell}_{\alpha}(\mP_3), \N^{cell}_{\alpha}(\mP_4)\}\leq \N^{cell}_{\alpha}(\mP_5)\leq \N^{cell}_{\alpha}(\mP_6)\leq \N^{cell}_{\alpha}(\mP_R).$$}
\end{remk}


\begin{remk}{\normalfont Since $\N^{cell}(\mP_R)$ contains all finite groups and any finite dimensional
contractible complex is in  $\mP_R$, $\N^{cell}(\mP_R)$ contains the class of $\sl H{\mathcal F}$-groups. }
\end{remk}

\begin{lem}{\label{sc:13}}{\normalfont  $\N^{cell}(\mP_6)$ and $\N^{cell}(\mP_R)$ are closed under countable directed unions,
amalgamated products, and HNN-extensions.}
\end{lem}

\begin{proof}  Let $\mathcal{C}$ denote either $\N^{cell}(\mP_6)$ or $\N^{cell}(\mP_R)$.
Any group that is either a countable directed union, an
amalgamated product, or HNN-extension of groups in $\mathcal{C}$ is a fundamental group of a graph of groups with
vertex groups in $\mathcal{C}$. Let $G$ be such a group for a graph of groups $Y$. Let $X$ be the universal covering tree of $Y$. Then $X$ is a 1-dimensional contractible $G$-CW-complex
such that each stabilizer subgroup is in $\mathcal{C}$. It follows that $G\in \mathcal{C}$.
\end{proof}

The following result shows that a property similar to $\mP_R$ induces the class of finite groups.

\begin{prop}{\label{sc:14}}{\normalfont  Let $\mathcal B$ be the set of CW-complexes where
$X\in \mathcal B$ if and only if there exists $m\in \mathbb N$ (depending on $X$) such that
$H_k(X)= \Z^m$, $k=\mbox{dim}(X)$.  Then $\N^{cell}(\mathcal B)=\N^{cell}_1(\mathcal B)$
is the class of finite groups.}
\end{prop}
\begin{proof}  $\mP_2\leq \mathcal B$ implies $\N^{cell}_1(\mP_2)\subseteq \N^{cell}_1(\mathcal B)$.

Conversely, we use induction on the ordinals to show that every group in $\N^{cell}_{\beta}(\mathcal B)$ is finite for each $\beta$.

Let $G\in \N^{cell}_{\beta}(\mathcal B)$. Then $G$ acts on $X$ with property $\mathcal B$ such that for each cell $\sigma$
there exists an ordinal $\alpha < \beta$ such that the stabilizer subgroup $G_{\sigma}$ is finite.
The following lemma finishes the proof.
\end{proof}

\begin{lem} {\label{sc:15}}{\normalfont Suppose $X$ is an $n$-dimensional $G$-CW-complex with $H_n(X)=\Z^m\oplus F$ for some positive integer $m$
and a finite group $F$ such that each stabilizer subgroup is finite. Then, $F=0$ and $G$ is finite.}
\end{lem}
\begin{proof} Suppose $F\ne 0$. Then, there exists a cycle $\tau\in Z_n(X)=H_n(X)$ such that $[\tau]\in F$, $[\tau]\ne 0$. Since $|F|\cdot\tau=0$,
we have $\tau=0$ which is a contradiction.

For each $1\leq i \leq m$, let $$\tau_i = \displaystyle{\sum_{j={i_1}}^{i_{p_i}} k_{ij} \sigma_{ij}}$$ be a generator of the $i$th-factor of $Z_n(X)=\Z^m$,
 where each $k_{ij}$ is a positive integer and $\sigma_{ij}$ is an $n$-cells of $X$.
Consider the subspace $Y$ of $X$ formed by the $n$-cells defining these cycles, $$Y=\displaystyle{\bigcup_{1\leq i \leq m, 1\leq j \leq {p_i}} \sigma_{ij}}.$$ The action of $G$ on $X$ induces an action on $Z_n(X)=\Z^m$. Therefore, $G$ acts on $Y$ by permuting the $n$-cells.
Since the stabilizers are finite, this implies that $G$ must also be finite.
%
\end{proof}
%
%
%

\section{Jump (Co)homology}

In \cite{nans}, we considered a new (co)homological condition for groups
called jump (co)homology. In this section, we take a closer look at
groups with this property.

\begin{defn}{\normalfont Let $R$ be a commutative ring with a unit. A discrete group $G$ has
{\it jump (co)homology over $R$} if there exists an
integer $k\geq 0$, such that for each subgroup $H$ of $G$ we have
$hd_R(H)=\infty$ ($cd_R(H)=\infty$) or $hd_R(H)\leq k$
($cd_R(H)\leq k$).  The smallest of all such $k$ will
 be called {\it jump height} and denoted $hjh_R(G)$ ($cjh_R(G)$). }
\end{defn}

\begin{remk}{\normalfont When $R=\Z$, we will simply say that $G$ has jump (co)homology with jump height $hjh(G)$ ($cjh(G)$).}
\end{remk}

It follows directly that if a group $G$ has finite virtual (co)homological
dimension over $R$, then it must have jump (co)homology of height $vhd_R(G)$ ($vcd_R(G)$) over $R$.

The converse is clearly not true, as it is evident from the example of the group $\mathbb Q/\Z$. This is an infinite
 torsion group which has jump (co)homology of height zero, but it does not have finite virtual
 (co)homological dimension.


In the next proposition,  we show that groups with jump (co)homology satisfy relations similar to
those of groups with finite (co)homological dimension. Whenever the proofs of homological and cohomological
parts are analogous, we only prove one and omit the proof of the other.

\begin{prop}{\label{sc:3.1}}{\normalfont Let $G$ be a group and let $R$ be a commutative ring with a unit.
\begin{enumerate}
\item If $G$ has jump (co)homology over $R$, then every subgroup $S<G$
has jump (co)homology over $R$ with $hjh_R(S)\leq hjh_R(G)$ ($cjh_R(S)\leq cjh_R(G)$).

\vspace{1mm}

\item If $G$ has jump (co)homology  of height zero over $R$, then $G$ is a torsion group. Conversely,
 if $G$ is an $R$-torsion group, then $G$ has jump (co)homology of height zero over $R$.

\vspace{1mm}

\item If $G$ has jump cohomology over $R$, then $G$ has jump homology over $R$ and $hjh_R(G)\leq cjh_R(G)$. If $G$ is countable, then
$G$ has jump cohomology over $R$ if and only if $G$ has jump homology over $R$ and $hjh_R(G)\leq cjh_R(G)\leq hjh_R(G)+1$.
\end{enumerate}}
\end{prop}

\begin{proof}

\smallskip

\noindent {\it(1):} By definition, jump cohomology is a subgroup closed property.

\smallskip

\noindent {\it(2):}
 The first part of the claim follows from the fact that an infinite cyclic group has (co)homological dimension one over $R$.

 Now, let $G$ be an $R$-torsion group and let $H$ be a nontrivial subgroup generated by a single element of $G$. By assumption, $|H| \cdot 1_R$ is not
 invertible in $R$. Therefore, $cd_R(H)=hd_R(H)=\infty$. This shows that any nontrivial subgroup of $G$ has infinite (co)homological dimension over $R$.

\smallskip

\noindent {\it(3):}  Let $G$ be countable and let $H<G$. Then, by Theorem 4.6 of \cite{bieri}, we have
$hd_R(H)\leq cd_R(H)\leq hd_R(H)+1$. This proves $hjh_R(G)\leq cjh_R(G)\leq hjh_R(G)+1$.

In general, suppose $G$ has jump cohomology over $R$. Let $H<G$ so that $hd_R(H)<\infty$. Since
homology of groups commutes with direct limits, we can find a finitely generated subgroup $H'< H$ such that $hd_R(H')=hd_R(H)$.
Then, $hd_R(H')\leq cd_R(H')\leq hd_R(H')+1$ implies $hd_R(H)\leq cjh_R(G)$.
\end{proof}

\begin{remk}{\label{sc:3.2}}{\normalfont Note that over the integers, part $(2)$ of the proposition states that a group
is torsion if and only if it has jump (co)homology of height zero.

More generally, a finite group $G$ has (co)homological dimension zero only if it is $R$-torsion-free. If $G$ contains
 $R$-torsion, then $cd_R(G)=hd_R(G)=\infty$. This implies that every finite group has jump (co)homology of height zero over $R$.
 It follows that every locally finite group has jump homology of height zero over $R$.

So, we see that relaxing the definition of (co)homological dimension of a group allows us to consider groups with large torsion subgroups.}
\end{remk}

\begin{prop}{\label{sc:3.3}}{\normalfont Let $G$ be a group and let $R$ be a commutative ring with a unit.
\begin{enumerate}

\item Let $G$ be a direct limit of (countable) groups $G_i$, $i\in I$, where
each $G_i$ has jump (co)homology over $R$. If  $k=\mbox{sup}\{hjh_R(G_i)\}<\infty$ $(k=\mbox{sup}\{cjh_R(G_i)\}<\infty),$
then $G$ has jump (co)homology of height $k$ ($k$ or $k+1$) over $R$.

\vspace{1mm}

\item If a (countable) group $G$ has jump (co)homology over $R$, then  there
exists a finitely generated subgroup $S<G$ with jump (co)homology over $R$ with $hjh_R(S)=hjh_R(G)$
($cjh_R(G)\leq cjh_R(S)\leq cjh_R(G)+1$).

%

\end{enumerate}}
\end{prop}

\begin{proof}
\smallskip

\noindent {\it(1):} Let $H<G$ such that $cd_R(H)<\infty$ and let $H_i=G_i\cap H$. Then, $H=\underrightarrow{\lim}_{i\in I} H_i$ and
the theorem of Berstein (see Thm. 4.7 in \cite{bieri}) shows
$$cd_R(H)\leq \mbox{sup}\{cd_R(H_i)\}+1\leq k+1.$$ Since $k=cjh_R(G_j)$ for some $j\in I$, we have $cjh_R(G)=k \mbox{ or } k+1$.

For the homological part, let $H<G$ with $hd_R(H)<\infty$ and let $H_{\alpha}=G_{\alpha}\cap H$.
Then, $H=\underrightarrow{\lim}_{{\alpha}\in I} H_{\alpha}$ and  $$H_k(H,M)=\underrightarrow{\lim}_{{\alpha}\in I}H_k(H_{\alpha},M)$$ for each $k$ and a $\Z H$-module $M$. It
follows that $hd_R(H)\leq k$ and thus $hjh_R(G)\leq k$. Since $k=hjh_R(G_{\beta})$ for some ${\beta}\in I$, $hjh_R(G)=k$.

\smallskip

\noindent {\it(2):} This is a direct application of (1).

%
%
%
\end{proof}

In the next proposition, we derive relations between jump heights of  groups forming a short exact sequence.

\begin{prop}{\label{sc:3.4}}{\normalfont Let $G$ be a group and let $R$ be a commutative ring with a unit.
\begin{enumerate}
\item Suppose $H$ is a finite index subgroup of $G$. Then $G$ has jump (co)homology over $R$
if and only if $H$ has jump (co)homology over $R$. Moreover,
 $hjh_R(G)=hjh_R(H)$ ($cjh_R(G)=cjh_R(H)$).

\vspace{1mm}

\item Suppose $$1\rightarrow T {\buildrel \iota \over \longrightarrow} G  {\buildrel \pi \over \longrightarrow} Q\rightarrow 1$$
is a short exact sequence of groups
such that $Q$ has jump (co)homology  and $T$ is  $R$-torsion. Then $G$ has jump
(co)homology of height $hjh_R(G)\leq hjh_R(Q)$ ($cjh_R(G)\leq cjh_R(Q)$).

\vspace{1mm}

\item Suppose $$1\rightarrow K {\buildrel \iota \over \longrightarrow} G  {\buildrel \pi \over \longrightarrow} Q\rightarrow 1$$
 is a short exact sequence of groups such that $Q$ has finite virtual (co)homological dimension over $R$ and $K$ has jump (co)homology
 over $R$. Then $G$ has jump (co)homology of
height $$hjh_R(G)\leq hjh_R(K)+vhd_R(Q)$$ $$(cjh_R(G)\leq cjh_R(K)+vcd_R(Q))$$ over $R$.
\end{enumerate}}
\end{prop}

\begin{proof}
\noindent {\it(1):}
 Suppose that $H$ has jump
cohomology over $R$. Let $T<G$ with $cd_R(T)<\infty$. Since $T\cap H$ is a finite index subgroup of $T$,
$cd_R(T\cap H)=cd_R(T)$. Therefore, $cd_R(T)\leq cjh_R(H)$ and $cjh_R(G)=cjh_R(H)$.

\smallskip

\noindent {\it(2):} Let $H<G$ with $cd_R(H)<\infty$.  Since $\pi|_{H}:H\to Q$ is injective, it follows
 that $cjh_R(G)\leq cjh_R(Q)$.

\smallskip

\noindent {\it(3):} By (1), we can assume $cd_R(Q)<\infty$. Let $H<G$ with $cd_R(H)<\infty$.
 Since $$cd_R(H)\leq cd_R(\iota^{-1}(H))+ cd_R(\pi(H)),$$ we have $$cd_R(H)\leq cjh_R(K)+cd_R(Q).$$

\end{proof}

\begin{remk}{\label{sc:3.5}} {\normalfont The cohomological part of the inequality in (3) cannot be generalized for short exact sequences where
the group $Q$ has jump cohomology, but not necessarily finite virtual cohomological dimension.

For example, consider the short exact sequence
$$0\rightarrow \Z\longrightarrow  \mathbb Q \longrightarrow  \mathbb Q/{\Z}\rightarrow  0.$$
Note that $cjh(\Z)=cd(\Z)=1$, $cjh(\mathbb Q/{\Z})=0$, and yet,
$cjh(\mathbb Q)=cd(\mathbb Q)=2$.}
\end{remk}

\medskip

There are some interesting properties of jump (co)homology for the classes of  solvable and linear groups which we present next.

\subsection{Solvable groups}

In what follows, let $R$ be a commutative integral domain of characteristic zero. First, we recall
a well-known result.

\begin{thm}{\label{sc:3.6}}{\normalfont (Stammbach, \cite{stam}) Let $G$ be an
$R$-torsion-free solvable group with Hirsch length $h<\infty$. Then, $hd_R(G)=h$
and there exist an $\mathbb FG$-module $A$ additively isomorphic to the fractional
field ${\mathbb F}$ of $R$, such that $H^{\mathbb F}_h(G, A)\cong {\mathbb F}$.}
\end{thm}

\begin{lem}{\label{sc:3.7}}{\normalfont Let $G$ be a virtually polycyclic group. Then $G$ has jump (co)homology
 over $R$ and  $$hjh_R(G)=cjh_R(G)=vcd_R(G)=h(G).$$}
\end{lem}
\begin{proof} Let $H$ be a finite index poly-$\Z$ subgroup of $G$.
Then, $hd_R(H)=h(G)$. Since, $hd_R(H)\leq cd_R(H)\leq h(H)$, we have
 $cd_R(H)=hd_R(H)=h(G)$.
\end{proof}

\begin{lem}{\label{sc:3.8}}{\normalfont Let $G$ be a nilpotent group. $G$ has jump homology
 over $R$ if and only if it has finite Hirsch length. In addition,
$hjh_R(G)=h(G)$.}
\end{lem}
\begin{proof} Suppose $G$ has jump homology over $R$.
Let $H$ be a finitely generated subgroup of $G$ such that $hd_R(H)=hjh_R(G)$. Then,
$h(H)=vcd_R(H)= hjh_R(G)$. This proves that $h(G)= hjh_R(G)$.

Now, suppose $h(G)<\infty$.
Let $H<G$ with $hd_R(H)<\infty$. By Stammbach's Theorem, $ hd_R(H)= h(H)$.
This shows that $hjh_R(G)\leq h(G)$.
\end{proof}


\begin{thm}{\label{sc:3.9}}{\normalfont Let $G$ be a solvable group.

\smallskip

\begin{enumerate}
\item If $h(G)<\infty$, then $G$ has jump homology
over $R$ with $hjh_R(G)\leq h(G)$. Conversely, if $G$ is finitely generated with jump homology
over $R$, then  $h(G)< \infty$.
\smallskip
\item Let $\mathbb F$ be the fraction field of $R$. The conditions that $G$ has jump homology over $\mathbb F$,
$G$ has finite homological dimension over $\mathbb F$, and $G$ has finite Hirsch length are equivalent. Moreover, $$hjh_{\mathbb F}(G)= hd_{\mathbb F}(G)=h(G).$$
\end{enumerate}}
\end{thm}
\begin{proof} {\it(1):} Following the proof of Lemma \ref{sc:3.8}, we have that if $h(G)<\infty$, then $hjh_R(G)\leq h(G)$.

Suppose now $G$  is finitely generated and has jump homology over $R$.  Let
$G_1$ denote the commutator subgroup. Since $h(G)=h(G_1) + rk(G/G_1)$, by Lemma \ref{sc:3.8},
$h(G)=hjh_R(G_1)+rk(G/G_1)<\infty$.

\smallskip

\noindent {\it(2):} Since every group is $\mathbb F$-torsion-free, this is an easy
application of Stammbach's Theorem.
\end{proof}

%
%
%

\subsection{Linear groups}
Let $R$ is a commutative integral domain of characteristic zero.

\begin{lem}{\label{sc:3.10}}{\normalfont
Let $G$ be a finitely generated linear group. Then $G$ has jump (co)homology over $R$
 if and only if $G$ has finite virtual cohomological dimension.  In addition, $$vcd(G)=cjh(G)\geq cjh_R(G)=vcd_R(G)\geq hjh_R(G)=vhd_R(G).$$}
\end{lem}
\begin{proof}Suppose $vcd(G)<\infty$. Let $H<G$ with $[G:H]<\infty$ such that $cd(H)=vcd(G)$. Clearly, $cd(H)=cjh(G)$.
Since $cd_R(H)\leq cd(H)$, we have $cjh(G)\geq cjh_R(G)$.
By (3) of Proposition \ref{sc:3.1}, it also follows $cjh_R(G)\geq hjh_R(G)$.

Now, suppose $G$ has jump cohomology over $R$.  According to a deep theorem of Alperin and Shalen (see \cite{alperin}) $G$ has finite virtual cohomological dimension if and only if there is an upper bound on the Hirsch lengths of its finitely generated unipotent subgroups.
 Let $U$ be a finitely generated unipotent subgroup of $G$. From Lemma \ref{sc:3.8}, $h(U)=cjh_R(U)\leq cjh_R(G)$. This shows $vcd(G)<\infty$.

 The homological part is analogous.
\end{proof}

The notion of jump (co)homology allows us to extend the Alperin's and Shalen's result to linear groups which are not necessarily finitely generated.

\begin{thm}{\label{sc:3.11}}{\normalfont A (countable) linear group $G$ has jump (co)homology over $R$ if and only if there is a
finite upper bound on the Hirsch lengths of its finitely generated unipotent subgroups.}
\end{thm}
\begin{proof} Suppose $G$ has jump
cohomology over $R$. Let $U$ be a finitely generated unipotent subgroup of $G$. Then,
$h(U)=cjh_R(U)\leq cjh_R(G)$. The proof of the homological part is similar.

For the converse, suppose the Hirsch lengths of all finitely generated unipotent subgroups is bounded by an integer $k$.
We will show that $G$ has jump homology over $R$.
Let $G$ be a direct limit of its finitely generated subgroups $G_{\alpha}$, $\alpha\in I$. By (1) of Proposition \ref{sc:3.3},
it is enough to show that each $G_{\alpha}$ has jump homology over $R$ and $\mbox{sup}\{hjh_R(G_{\alpha})\}<\infty$. According to
the remark after Theorem 3.3 of \cite{alperin}, $vcd(G_{\alpha})\leq 1+C+N_{G_{\alpha}}+ mn+ Cn^4,$
for each ${\alpha}\in I$. Here, the only constant that depends on $G_{\alpha}$ is $$N_{G_{\alpha}}=\mbox{sup}\{cd(U_{\beta})| U_{\beta}
\mbox{ a unipotent subgroup of } G_{\alpha}\}.$$ Therefore, $vcd(G_{\alpha})\leq 1 + C + k + mn + Cn^4$  and
$hjh_R(G_{\alpha}) = vhd_R(G_{\alpha})
                    \leq 1 + C + k + mn + Cn^4,$
for each ${\alpha}\in I$.

The cohomological version now follows by \ref{sc:3.1}.
\end{proof}

\begin{cor}{\label{sc:3.12}}{\normalfont Suppose $$\displaystyle{1\rightarrow S {\buildrel\iota \over \longrightarrow} G  {\buildrel \pi \over \longrightarrow} Q\rightarrow 1}$$
 is a short exact sequence of (countable) linear groups
such that $S$ and $Q$ have jump (co)homology over $R$. Then $G$ has jump (co)homology over $R$ and
$$hjh_R(G)\leq hjh_R(S)+ hjh_R(Q)$$ $$(cjh_R(G)\leq cjh_R(S)+ cjh_R(Q)+1).$$}
\end{cor}
\begin{proof}
Let $H$ be a finitely generated subgroup of $G$.
Let $S_1=\iota^{-1}(H)$ and $Q_1= \pi(H)$. By \ref{sc:3.4}, $hjh_R(H)\leq hjh_R(S_1)+vhd_R(Q_1)\leq hjh_R(S)+hjh_R(Q)$. Now, (1) of Proposition \ref{sc:3.3} concludes the proof.
\end{proof}

\medskip

\subsection{Groups without $R$-torsion}
Here we assume more generally that $R$ is a commutative ring with a unit. In \cite{nans}, we studied the following

\begin{conj}{\label{sc:3.13}}{\normalfont (\cite {nans}) Let $G$ be a group without $R$-torsion and let $k$ be nonnegative
integer. Then $G$ has jump cohomology of height $k$ over $R$ if and only if $G$ has finite cohomological
dimension $k$ over $R$.}
\end{conj}

It is immediate that if $cd_R(G)=k$, then $cjh_R(G)=k$.
In \cite{nans}, we proved that the converse holds when $G$ is in
 ${\sl H}{\mathcal F}$.

An interesting type of groups that have jump cohomology are groups with periodic cohomology.
Namely, a group $G$ has {\it periodic cohomology after $k$-steps over $R$} if there
exists an integer $q\geq 0$ such that the functors  $H^{i}_R(G, -)$
and $H^{i+q}_R(G, -)$ are naturally equivalent for all $i>k$.

For torsion-free groups, it is has been conjectured by Talelli (see \cite{tal2})
that the notions of periodic cohomology and finite cohomological dimension are equivalent. This can
be seen as a special case of Conjecture \ref{sc:3.13}.

A deep theorem of Adem and Smith shows the relevance of Talelli's Conjecture in understanding free group actions
on homology spheres.

\begin{thm}{\label{sc:3.14}} {\normalfont (Adem-Smith, 2001, \cite{AS}) A group $G$ has
periodic cohomology induced by a cup product map if and only if $G$ acts freely and properly on
a finite dimensional complex homotopy equivalent to a sphere. If, in addition,
$G$ is countable, then it acts freely, properly discontinuously, and smoothly on some $S^n\times \R^k$.}
\end{thm}

Talelli's Conjecture asserts that  a torsion-free group that acts freely and properly
 discontinuously on some $S^n\times \R^k$, in fact, can act freely and properly discontinuously on a
 Euclidean space.


%

\section{Identifying $\N^{cell}_1(\mP_R)$}

The subclass $\N^{cell}_1(\mP)\subseteq \N^{cell}(\mP)$ consists of groups that act
freely and cellularly on a finite dimensional CW-complex in $\mP$. If $\mP$ poses
no additional restrictions, then $\N^{cell}_1(\mP)=\N^{cell}(\mP)$
is the class of all groups. This follows from the fact that any group can be realized as a fundamental group
of a 2-dimensional CW-complex. On the other hand, if we consider the class
of polycyclic groups $\N^{cell}(\mP_4)$, then $\N_1(\mP_4)$ is exactly the class of
cyclic groups. Thus, identifying the subclass $\N^{cell}_1(\mP)$ can be seen as a first step in understanding
properties of groups in $\N^{cell}(\mP)$.
In this regard, we proceed to show that every group in $\N_1(\mP_R)$ has jump (co)homology over $R$.

First we need a proposition which is a generalization of a result obtained in \cite{nans} (see Prop. 2.5).

\begin{prop}{\label{sc:3.15}}{\normalfont Suppose $R$ is an integral domain of characteristic zero. Let $G$ be a group and  let $X$ be
an $n$-dimensional $G$-CW-complex. Suppose there exist an integer
$k$ such that for
each $i > k$, $H_i(X, \mathbb Z)$ is an $R$-torsion-free torsion group and $H_k(X, \mathbb Z)\cong {\mathbb Z}^m \oplus F$ where $m>0$
and  $F$ is an $R$-torsion-free finite group.
If all the stabilizer subgroups of the action of $G$ on $X$ have jump (co)homology over $R$
uniformly bounded by an integer $b$, then $G$ has jump (co)homology  over $R$ with height at most $ b + n - k +
\displaystyle { {1\over 2}{m(m-1)}}.$}
\end{prop}
\begin{proof}We prove the cohomological statement and note that the homological part is analogous.

The action of $G$ on the complex $X$ induces an action on $H_k(X,
\mathbb Z)$, which may be nontrivial. Since $H_k(X, \mathbb
Z)\cong {\mathbb Z}^m\oplus F$,
$\mbox{GL}_m(\Z)$ is a finite index subgroup of $\mbox{Aut}(H_k(X,
\mathbb Z))$. Thus, we can find a finite index subgroup of $G$
that acts separately on the two summands $\mathbb Z^m$  and $F$,
and acts trivially on $F$. In view of (1) of Proposition \ref{sc:3.4}, we can
assume that $G$ acts this way. Then, the action is induced by a
representation of $\rho :G\to \mbox{GL}_m(\Z)$.

Suppose, by contradiction, that $G$ has a subgroup of finite cohomological dimension over $R$ larger than
our estimate. Without loss of generality, we  can assume this for $G$ itself. Let $H=\mbox{ker}(\rho)$.  Since $$vcd_R(\rho(G))\leq vcd(\mbox{GL}_m(\Z))=\displaystyle{{1\over 2}{m(m-1)}},$$ we have $cd_R(H)> b + n - k$.

Let $h=cd_R(H)$. Then,
$H_R^i(H, M)= 0$ for all $i> h$  and all ${R}H$-modules $M$, and there exist a ${R}H$-module $A$ such
that $H_R^{h}(H, A)\ne 0$.

Consider the double complex $\displaystyle{\mbox{Hom}(P_*,
C^*(X,A))}$, where $P_*$ is a projective resolution of $R$ over ${R}H$ and $C^*(X,A)$ is the cellular co-chain
complex of $X$ with ${R}H$-module coefficients $A$. The
natural bi-grading of the complex gives us the
spectral sequence,
$$E^{p,q}_2(A)= H_R^p(H, H^q(X, A))
\Longrightarrow H_R^{p+q}(H, C^*(X, A)).$$
Observe that $E^{p,q}_2(A) = 0$ when $p> h$ or $q> k$. The corner argument then
shows,
$$ H_R^{h+k}(H, C^*(X,A))\cong
H_R^h(H, H^k(X, A)).$$
By the Universal Coefficient Theorem, we also have
$$0\rightarrow \mbox{Ext}_{\Z}^1(H_{k-1}(X),\Z)\rightarrow H^k(X, A)\rightarrow \mbox{Hom}(H_{k}(X),A)\rightarrow 0.$$
The associated long exact sequence in cohomology shows that $H_R^h(H, (H^k(X,A))$ surjects onto $H_R^h(H, \mbox{Hom}(H_{k}(X),A))$.
On the other hand,
\begin{align*}
H_R^h(H, \mbox{Hom}(H_{k}(X),A)) &\cong H_R^h(H, A^m)\\
                                    &\cong {\underbrace{{H_R^h(H, A)}\oplus \dots \oplus{H_R^h(H, A)}}_{\mathrm{m}}}\\
                                    &\ne 0.\\
\end{align*}
Therefore, $H_R^{h+k}(H, C^*(X,A))\ne 0$, showing that $$\mbox{dim}(H_R^{*}(H, C^*(X,A)))= h+k > b+n.$$

The double complex also gives the $E_1$-term spectral sequence,

$$E^{p,q}_1(A)=\displaystyle{\bigoplus_{\sigma \in \Sigma_p}}
H_R^q(H_{\sigma}, A^{\sigma})\Longrightarrow
H_R^{p+q}(H, C^*(X, A)),$$
where $X_p$ is the collection of all the $p$-cells and $\Sigma_p$ denotes
a set of representatives of all the $H$-orbits in $X_p$.

According to our hypotheses, $E^{p,q}_1 = 0$
when $p>n$ or $q>b$. This implies that   $$\mbox{dim}(H_R^{*}(H, C^*(X,
A)))\leq b+n,$$ contradicting our earlier estimate.

\end{proof}

\begin{thm}{\label{sc:3.16}}{\normalfont
Let $\mathcal{J_R}$ be the class of groups with jump cohomology over $R$
and let $\mathcal{VCD}$ denote the class of groups with finite virtual cohomological dimension.
Then, $$\mathcal{VCD}\subseteq \N^{cell}_1(\mP_R)\subseteq \mathcal{J_R}.$$}
\end{thm}

\begin{proof} If $G\in \N^{cell}_1(\mP_R)$, then there is a finite dimensional free-$G$-CW-complex $X$
satisfying the hypotheses of \ref{sc:3.15}. So, $G$ has jump cohomology.

Suppose $G\in \mathcal{VCD}$. Let $N\lhd G$ such that $[G:N]< \infty$ and $cd(N)<\infty$.
Let $X$ be  a finite dimensional
contractible free-$N$-CW-complex. Proceeding as in the proof of Proposition \ref{sc:4}, we can construct a finite dimensional,
contractible CW-complex
$Y=\mbox{Map}_N(G, X)$. Since $N$ acts freely on $Y$, the intersection of each isotropy group of the action of
$G$ on $Y$ with $N$ is trivial.

Let $F=G/N$. Then $F$ acts freely and cellularly on a torus $T$. Let $G$ act on $T$
through the epimorphism of $G$ onto $F$ and let $G$ act on the product $Y\times T$ through the diagonal action.
We claim that this action is free. Indeed, the isotropy groups of the action of $G$ on $T$ are $N$ and the isotropy
groups of the action of $G$ on $Y$ intersect trivially with $N$.
\end{proof}

\begin{remk}{\label{sc:3.17}}{\normalfont Let $\mathcal T$ be the class consisting of only the trivial group. Observe that Conjecture \ref{sc:3.13}  implies that the class of torsion-free $\N^{cell}_1(\mP_{\Z})$-groups
 is exactly the class of $\sl H_1 \mathcal T$-groups, i.e. torsion-free $\sl H_1 \mathcal F$-groups.
In \cite{nans}, we have shown that  ${\sl H \mathcal T} \cap {\mathcal J_{\Z}}= {\sl H_1 \mathcal T}$.
Combining this together with \ref{sc:3.16}, implies $${\sl H \mathcal T} \cap \N^{cell}_1(\mP_{\Z})= {\sl H_1 \mathcal T}.$$}
\end{remk}

\begin{remk}{\label{sc:3.18}}{\normalfont It is known that the Thompson's group $F$ is of type  $\mbox{FP}_\infty$ but
not $\sl H \mathcal F$ (see \cite{BG}, \cite{krop}). This group can be
defined by the presentation,
$$\langle x_0, x_1, x_2, ... | x_i^{-1} x_n x_i = x_{n+1} \mbox{ for all }
i<n \mbox{ and } n\in {\mathbb N}\rangle.$$
We point out that $F$ does not have jump  (co)homology over $R$, because it has an infinite
rank abelian subgroup $${\mathbb Z}^\infty = \langle x_0
x_1^{-1}, x_2 x_3^{-1},  x_4 x_5^{-1}, ... \rangle.$$
This shows that
$\N^{cell}_1(\mP_R)$ does not contain $F$.}
\end{remk}


%
%

We end this section by an application of Proposition \ref{sc:3.15} to solvable groups.

\begin{thm}{\label{sc:3.19}}{\normalfont Let $G$ be a solvable group and let $X$ be a finite dimensional
$G$-CW-complex. Suppose there exists an integer $t$ such that $\displaystyle{\oplus_{i\geq t} H_i(X)}$ is finitely generated and infinite.
Then $G$ has finite Hirsch length if and only if there is an upper bound on the Hirsch lengths of  all the stabilizer subgroups.}
\end{thm}

\begin{proof} Let $b$ be a bound on the Hirsch lengths of  all the stabilizer
subgroups. By Theorem \ref{sc:3.9}, every stabilizer subgroup has jump homology over $\mathbb Q$ with
jump height equal to its Hirsch length. It can now be easily seen that
 the group $G$ and  the complex $X$ satisfy the hypotheses of
\ref{sc:3.15} for $R=\mathbb Q$. Therefore, $h(G)=hjh_{\mathbb Q}(G)<\infty$.
\end{proof}


$${\bf Acknowledgment}$$
I owe many thanks to Alain Valette for fruitful conversations about group actions on trees and for recalling the result of Lemma \ref{sc:7}.

%
%
%
%
%
%
%
%
%
%

\end{document}